\documentclass[11pt]{elsarticle}
\usepackage{geometry}
\usepackage{amsthm}
\usepackage{amssymb}
\usepackage{hyperref}
\usepackage[table]{xcolor}
\usepackage{wrapfig}
\usepackage{pst-all}
\usepackage{csvsimple}
\usepackage{caption}
\usepackage{csquotes}

\newtheorem{theorem}{Theorem}
\newtheorem{corollary}{Corollary}[theorem]

\theoremstyle{definition}
\newtheorem{definition}{Definition}
\newtheorem*{remark}{Remark}

\hypersetup{
    colorlinks=true,
    linkcolor=cyan,
    filecolor=magenta,      
    urlcolor=blue,
}

\definecolor{darker_green}{RGB}{0, 180, 0}
\definecolor{weak_green}{RGB}{150, 220, 150}
\definecolor{calmer_blue}{RGB}{80, 80, 180}
\definecolor{calmer_red}{RGB}{220, 10, 10}
\definecolor{weak_red}{RGB}{220, 110, 110}
\definecolor{grey}{RGB}{170, 170, 170}

\makeatletter
\newcommand{\thickhline}{%
    \noalign {\ifnum 0=`}\fi \hrule height 1pt
    \futurelet \reserved@a \@xhline
}
\newcolumntype{"}{@{\hskip\tabcolsep\vrule width 1pt\hskip\tabcolsep}}
\makeatother
\newcolumntype{R}[1]{>{\raggedright\arraybackslash}m{#1}}
\newcolumntype{C}{>{$}c<{$}}

\title{An efficient model for the preemptive single machine scheduling of equal-length jobs}
\author[1]{Artem Fomin\corref{email1}\fnref{fn1}}
\author[1]{Boris Goldengorin\corref{cor1}\fnref{fn2}}
\ead{goldengorin@gmail.com}
\cortext[cor1]{Corresponding author}
\fntext[fn1]{email: fomin.ae@phystech.edu}
\fntext[fn2]{email: goldengorin.bi@mipt.ru}
\address[1]{Department of Discrete Mathematics, Phystech School of Applied Mathematics and Informatics, Moscow Institute of Physics and Technology,  Institutsky lane 9, Dolgoprudny, Moscow region, 141700}
\date{}

\begin{document}

\newlength{\oldintextsep}
\setlength{\oldintextsep}{\intextsep}

\setlength\intextsep{0pt}

\normalsize

\begin{abstract}
We propose a Boolean Linear Programming model for the preemptive single machine scheduling problem with equal processing times, arbitrary release dates and weights(priorities) minimizing the total weighted completion time. Almost always an optimal solution of the Linear Programming relaxation is integral and can be straightforwardly converted into an optimal schedule. To deal with the fractional solutions we present two heuristics. Very often our heuristics find solutions with objective function values equal to the lower bound found by the Linear Programming relaxation. For the cases when upper bound returned by our heuristics differs from the lower bound we embed the bounds into a Branch and Bound algorithm, which solves the problem to optimality. Exhaustive computational study showed that the algorithm substantially surpasses state-of-the-art methods. 
\end{abstract}


\maketitle

\section{Introduction}
\label{sec:1}

We study the preemptive single machine scheduling problem with the Total Weighted Completion Time (TWCT) objective function. Based on the three field notation of \cite{Graham}, this problem is usually denoted as $1|pmtn; p_j = p; r_j|\sum w_jC_j$. The first field shows the number of machines, which is one in our case. The second indicates the job characteristics: $pmtn$ means that jobs can be preempted, $p_j = p$ indicates that all processing times are equal to $p$, and $r_j$ shows that jobs have arbitrary release times. In the third we define the objective function to be minimized. Here $w_j$ and $C_j$ are the weight and completion time of job $j$ respectively, and $\sum w_j C_j$ is the total weighted completion time function. According to \cite{compl} the problem's complexity status is still unknown. 

The literature devoted to the exact solution of our problem is rather scarce and we refer the reader to the recent publications \cite{Jar_Erk_2017} for the equal-length jobs case and \cite{Jar_Erk_2020} for the general case.

We start our paper by formulating and proving several properties of optimal schedules. All properties apply to a wide range of objective functions and most of them are valid for non-equal processing times. Some of the properties were previously studied in \cite{properties} and \cite{WSRPT}.

Based on our optimal schedule properties we formulate a Boolean Linear Programming (BLP) model to solve the scheduling problem. Our BLP model proved itself to be quite applicable, as its LP relaxation almost always returns an optimal integral solution, which can easily be converted into a feasible optimal schedule. For the rare cases when an optimal solution to the LP relaxation is fractional we propose two heuristics. We embed our two heuristics into the Branch and Bound (BnB) algorithm with the aforementioned LP relaxation solution as the lower bound to verify the optimality of the returned heuristic solutions.

We perform an exhaustive computational study to evaluate heuristics' and exact BnB algorithm's performance. Exact algorithms far surpasses previous state-of-the-art methods (\cite{Jar_Erk_2017}) in both CPU times and size of the solved scheduling problems. Our two heuristics also show great results, outperforming the accuracy of the counterpart heuristics \cite{lp_heuristics} and \cite{WSRPT}.

We also show that our model is flexible and can be used for many different objective functions used in single machine scheduling, which include, but are not limited to Weighted Number of Tardy Jobs or Total Weighted Tardiness. Deadlines and unavailability periods can also easily be incorporated in the model. 

The rest of the paper is organized as follows. Several necessary properties of optimal schedules are proven in section \hyperref[sec:2]{2}. The BLP model is presented in section \hyperref[sec:3]{3}. Our heuristics, intuitions behind them and BnB algorithm are detailed in section \hyperref[sec:4]{4}. Computational study results are summarized in section \hyperref[sec:5]{5}. Section \hyperref[sec:6]{6} proposes various possible modifications of our BLP model. Section \hyperref[sec:7]{7} contains concluding remarks and future research directions.

\subsection{Problem formulation}

The problem $1|pmtn;r_j;p_j|\sum w_j C_j$ can be described as follows. We are given $n$ jobs that need to be processed. The jobs are processed on a single machine, which means that at each point in time no more than one job can be processed. Each job has a release time $r_j$, processing time $p_j$ and a priority weight $w_j$, which are all assumed to be integer. Job $j$ becomes available for processing at the moment $r_j$ and must be processed for the time $p_j$. If it is specified that $p_j=p$, it means that jobs have equal processing time $p$. Preemptions of one job in favor of another are allowed at any time and any number of times. A schedule is called feasible, if it distributes jobs' processing in time in a way that does not violate release times and each job $j$ is processed for exactly $p_j$ time. Earliest point in time, at which the job is completed --- processed for $p_j$ time --- is called completion time and denoted as $C_j$. The objective is to find a feasible schedule, which minimizes $\sum_{j=1}^n w_jC_j$ --- the Total Weighted Completion Time (TWCT).

\section{Properties of optimal schedules}
\label{sec:2}

In this section we formulate and prove several properties of optimal schedules. The properties are inherent for more general problem $1|pmtn; r_j, p_j|F(C_j)$, where processing times $p_j$ may differ and objective function $F(C_j)$ is any increasing function of completion times. The proofs will remain correct for non-decreasing functions as well, but only if statements are changed from "Optimal schedules have a property" to "There is an optimal schedule with a property". More on that at the end of the section. The only property specific to the equal-length jobs setting will show itself to be extremely powerful in modelling the scheduling problem. 

\begin{remark}
\renewcommand{\baselinestretch}{1.5} 
We will suppose that an optimal solution does not contain idle time intervals. Because if idle time interval is necessary, then there is a point in time, before which a set of jobs $A$ is fully processed, but jobs in $J\setminus A$ (where $J$ denotes the set of all given jobs) are not yet available for processing. Then the problem can be decomposed into two subproblems: the problem of scheduling jobs in $A$ and the problem of scheduling jobs in $J\setminus A$. And if an idle time interval is not necessary, then an optimal schedule obviously will not contain it. Or, if the function is non-decreasing, \emph{there is} a schedule, which does not contain idle time.
\end{remark}

\begin{theorem}
In an optimal schedule for the problem $1|pmtn; r_j, p_j|F(C_j)$ preemptions occur at unit points in time only.
\end{theorem}
\begin{proof}
We will prove the theorem by contradiction. Take an optimal schedule. Assume there are preemptions at fractional points in time. Choose the earliest such point in time $t$. Let $\underline{t} = \lfloor t \rfloor$ (greatest integer less than or equal to $t$) and $\overline{t} = \lceil t \rceil$ (least integer greater than or equal to $t$). As there is a preemption at point $t$, there are multiple jobs assigned to be processed in the time interval $[ \, \underline{t}, \overline{t}\, ]$. Select any two $i$ and $j$ out of them. Let the fractions of the jobs $i$ and $j$ processed in $[ \, \underline{t}, \overline{t}\, ]$ be $\lambda_i$ and $\lambda_j$ respectively. Since $t$ is the earliest fractional preemption point, only integral parts of jobs $i$ and $j$ were processed before the point $\underline{t}$, hence, integral parts are left for processing after $\underline{t}$. Which means that there is at least $1-\lambda_i$ processing of job $i$ and at least $1-\lambda_j$ processing of job $j$ done after $\overline{t}$. Without loss of generality, we will assume that completion time of job $i$ is less than completion time of job $j$. 

Then, we find latest (not necessarily continuous) $\lambda_j$ amount of job $i$'s processing which is done after $\overline{t}$, and swap it with job $j$'s processing in the $[ \, \underline{t}, \overline{t}\, ]$ interval. Finding that amount is always possible, because $\lambda_i + \lambda_j \leq 1$, since they fit in a unit time interval $[ \, \underline{t}, \overline{t}\, ]$. Therefore, $\lambda_j \leq 1 - \lambda_i$. And there is at least $1 - \lambda_i$ of job $i$'s processing after $\overline{t}$, hence, at least $\lambda_j$. And swapping is possible, because $r_i$ is integer, therefore $r_i \leq t \Rightarrow r_i \leq \underline{t}$. 

That swap does not affect job $j$'s completion time, because the parts swapped are strictly prior to it. And job $i$'s completion time is now earlier, because its latest processing intervals have been swapped to an earlier time. Which means that we have found a solution with a lesser objective function value than the optimal schedule has. Which brings us to contradiction.
\end{proof}

Now we will separate each job $j$ into $p_j$ job-parts $j.1, j.2, \ldots , j.p_j$ and time into time intervals $1, 2, \ldots , T$, where $t$ stands for the time interval $(t-1, t]$ and $T=\sum_{j=1}^n p_j$ is the time horizon.

\begin{corollary}
In an optimal schedule for the problem $1|pmtn; r_j, p_j|F(C_j)$ every job-part is entirely assigned to one time interval.
\end{corollary}
\begin{proof}
Immediately follows from Theorem 1, because a schedule with preemptions at unit points in time will be a schedule, in which every job-part is entirely assigned to one time interval. 
\end{proof}

This result allows us to reduce the scheduling problem to the problem of assigning job-parts to time intervals. 

\begin{definition}
In a specific schedule job $i$ \textit{intersects} job $j$, if there is job $i$'s processing assigned in-between of some of job $j$'s processing. The in-between processing is called \textit{intersecting}.
\end{definition}

\vspace{0.5cm}
\begin{figure}[h]
\captionbox[Caption]{Job $j$ is intersected by job $i$}{
\psset{labels=none, xunit=0.5cm, yunit=0.5cm, yAxis=false}	
\begin{pspicture}(0,0)(20,0)	
\psaxes[Dx=5, subticks=5]{->}(0,0)(0,0)(22,0)	
\pcline[linewidth=3pt, linecolor=calmer_red]{-}(6,0)(9,0)
\pcline[linewidth=3pt, linecolor=calmer_blue]{-}(11,0)(15,0)
\pcline[linewidth=3pt, linecolor=calmer_red]{-}(16,0)(20,0)
\uput{8pt}[90](7.5,0){$j$}
\uput{8pt}[90](13,0){$i$}
\uput{8pt}[90](18,0){$j$}
\end{pspicture}}
\end{figure}
\vspace{0.5cm}

\begin{definition}
In a specific schedule jobs $i$ and $j$ are \textit{intersecting}, if $i$ intersect $j$ and vice versa.
\end{definition}

\vspace{0.5cm}
\begin{figure}[h]
\captionbox[Caption]{Intersecting jobs $i$ and $j$}{
\psset{labels=none, xunit=0.5cm, yunit=0.5cm, yAxis=false}	
\begin{pspicture}(0,0)(20,0)	
\psaxes[Dx=5, subticks=5]{->}(0,0)(0,0)(22,0)	
\pcline[linewidth=3pt, linecolor=calmer_blue]{-}(1,0)(3,0)
\pcline[linewidth=3pt, linecolor=calmer_red]{-}(6,0)(9,0)
\pcline[linewidth=3pt, linecolor=calmer_blue]{-}(11,0)(13,0)
\pcline[linewidth=3pt, linecolor=calmer_red]{-}(16,0)(20,0)
\uput{8pt}[90](2,0){$i$}
\uput{8pt}[90](7.5,0){$j$}
\uput{8pt}[90](12,0){$i$}
\uput{8pt}[90](18,0){$j$}
\end{pspicture}}
\end{figure}
\vspace{0.5cm}

\begin{theorem} 
In an optimal schedule for the problem $1|pmtn; r_j, p_j|F(C_j)$ there are no intersecting jobs.
\end{theorem}
\begin{proof}
We will prove the theorem by contradiction. Assume that there are intersecting jobs $i$ and $j$ in an optimal schedule. Without loss of generality, we will assume that the job $i$'s completion time is less than job $j$'s. Then take some of job $j$'s intersecting processing and swap it with the same amount of job $i$'s latest processing. Any positive amount of processing will suffice. The intersecting processing can only be earlier than job $i$'s latest processing. And it is later than some other job $i$'s processing, hence, later than $r_i$ and such swap is possible. Job $j$'s completion time has not changed, because it is later than any of job $i$'s processing. But job $i$'s completion is now earlier, because we moved the latest processing to an earlier time. Which brings us to a contradiction to the optimality of the initial schedule.
\end{proof}

\vspace{0.7cm}
\begin{figure}[h]
\captionbox[Caption]{Theorem 2 swap example}{
\psset{labels=none, xunit=0.5cm, yunit=0.5cm, yAxis=false}	
\begin{pspicture}(0,0)(20,0)	
\psaxes[Dx=5, subticks=5]{->}(0,0)(0,0)(22,0)	
\pcline[linewidth=3pt, linecolor=calmer_blue]{-}(1,0)(3,0)
\pcline[linewidth=3pt, linecolor=calmer_blue]{-}(6,0)(7,0)
\pcline[linewidth=3pt, linecolor=calmer_red]{-}(7,0)(9,0)
\pcline[linewidth=3pt, linecolor=calmer_blue]{-}(11,0)(12,0)
\pcline[linewidth=3pt, linecolor=calmer_red]{-}(12,0)(13,0)
\pcline[linewidth=3pt, linecolor=calmer_red]{-}(16,0)(20,0)
\uput{8pt}[90](2,0){$i$}
\uput{8pt}[90](18,0){$j$}
\end{pspicture}}
\end{figure}
\vspace{0.5cm}

Fig. 3 provides an example of the swap described in Theorem 2, applied to the intersecting jobs $i$ and $j$ from Fig. 2.

\begin{corollary}
In an optimal schedule for the problem $1|pmtn; r_j, p_j|F(C_j)$ if job $i$ is partially processed between processing intervals of job $j$, then job $i$ is fully processed between these intervals.
\end{corollary}

\begin{proof}
If job $i$ is not fully processed, then either there is some of its processing before the earlier interval of job $j$ or some of its processing after the latter. Job $j$ is already intersected by job $i$. And in the first case, the earlier processing interval of job $j$ would be between processing intervals of job $i$; in the second case, the latter processing interval of job $j$ would be between processing intervals of job $i$. Either way, job $i$ would also be intersected by job $i$, which is impossible by Theorem 2.
\end{proof}

\vspace{1cm}
\begin{figure}[h!]
\captionbox[Caption]{Cor. 2.1 example}{
\psset{labels=none, xunit=0.5cm, yunit=0.5cm, yAxis=false}	
\begin{pspicture}(0,0)(20,0)	
\psaxes[Dx=5, subticks=5]{->}(0,0)(0,0)(22,0)	
\pcline[linewidth=3pt, linecolor=weak_green]{-}(1,0)(3,0)
\pcline[linewidth=3pt, linecolor=calmer_red]{-}(5,0)(8,0)
\pcline[linewidth=3pt, linecolor=darker_green]{-}(10,0)(12,0)
\pcline[linewidth=3pt, linecolor=calmer_red]{-}(15,0)(18,0)
\pcline[linewidth=3pt, linecolor=weak_green]{-}(19,0)(21,0)
\uput{8pt}[90](2,0){$i_1$}
\uput{8pt}[90](6.5,0){$j_1$}
\uput{8pt}[90](11,0){$i_0$}
\uput{8pt}[90](16.5,0){$j_2$}
\uput{8pt}[90](20,0){$i_2$}
\end{pspicture}}
\end{figure}
\vspace{0.5cm}

To clarify, consider an example shown in Fig. 4. $j_1$ is the earlier interval of job $j$, $j_2$ is the latter. $i_0$ is the part of job $i$, processed between intervals $j_1$ and $j_2$. $i_1$ and $i_2$ depict two possibilities: either there is some processing interval $i_1$ of job $i$ before the $j_1$ or some processing interval $i_2$ after $j_2$. It can easily be seen that both possibilities make jobs $i$ and $j$ intersecting. 

Moving on to our special case with equal processing times, there is a peculiar property, which does not apply to the general case.

\begin{corollary}
In an optimal schedule for the problem $1|pmtn; r_j, p_j=p|F(C_j)$ there are either 0, $p$, or a multiple of $p$ time intervals between the two time intervals to which one job's subsequent job-parts are assigned.
\end{corollary}

\begin{proof}
There are zero time intervals between two subsequent job-parts, if the job is not interrupted at that moment. If it is interrupted, then any jobs, which are processed in-between, are processed fully before the interrupted job is continued by Corollary 2.1. To be fully processed, each one of them requires $p$ time intervals. That is, $p$ time intervals, if there is only one job, or a multiple of $p$, if several.
\end{proof}

For non-decreasing objective functions of completion times from these proofs immediately follows that \emph{there is} an optimal schedule with these properties. The schedule created by swaps in the theorems may not be better than the starting optimal one, therefore, not contradicting its optimality. It is not strictly better, but still not worse. Hence, also optimal. Which means that we can not claim that all optimal schedules have that property, but there is at least one optimal solution, which has. And finding that solution means solving the problem. Therefore, a problem with non-decreasing objective function of completion times can also be reduced to the problem of assigning job-parts to time intervals.

\section{BLP model}
\label{sec:3}

Now, using theoretical insights described in previous section, we will formulate a Boolean Linear Programming (BLP) model of the $1|pmtn; p_j = p; r_j|\sum w_j C_j$ problem. The formulation was also heavily inspired by \cite{bounds} and previously described in \cite{article}. 

First, we introduce the indices, the decision variables and define other necessary objects.

\vspace{0.5cm}
\begin{tabular}{ll}
$T = np$ & Time horizon.\\
$j \in \{1, 2, \ldots, n\}$ & Job index.\\
$k \in \{1, 2, \ldots, p\}$ & Job part index.\\
$t \in \{1, 2, \ldots, T\}$ & Time interval index.\\
$a \in \{1, 2, \ldots, n\}$ & Index for subsets of time intervals of jobs.\\
$b \in \{1,2,\ldots,p\}$ & Index for subsets of time intervals of job parts.\\
\end{tabular}
\vspace{0.5cm}

Decision variables:
\[x_{jkt} = \left\{ \begin{array}{ll} 
1 & \mbox{if job $j$'s $k$-th part is assigned to time interval $t$}, \\
0 & \mbox{otherwise}. \end{array} \right. \]

The weights for parts $1,\ldots,p-1$ are given by
\[w_{jkt} = \left\{ \begin{array}{ll}
0 & \mbox{for all}\>\> j \>\mbox{and}\>\> r_j + k - 1 \leq t \leq T-p+k, \\
\infty & \mbox{otherwise}. \end{array} \right. \]

and for part $p$ by
\[w_{jpt} = \left\{ \begin{array}{ll}
w_j t & \mbox{for all}\>\> j \>\mbox{and}\>\>  r_j + p \leq t, \\
\infty & \mbox{otherwise}. \end{array} \right. \]

Now define a subset of time intervals, which will allow us to make use of Corollary 2.2:
\[\mathcal{T}_{ab} = \left\{ \begin{array}{ll}
\emptyset & \mbox{if}\>\>b=p\>\> \mbox{and}\>\> a = n,\\
\{b, b + p, b+2p, \ldots, b + (a-1)p \}&\mbox{otherwise.}
\end{array} \right.
\]

The BLP model is formulated as follows:

\begin{eqnarray}
\min && \sum_{j=1}^{n} \sum_{k=1}^{p} \sum_{t=1}^{T} w_{jkt} x_{jkt}  \nonumber\\
\mbox{subject to}\>\>\quad \sum_{t=1}^{T} x_{jkt} &=& 1 \quad\quad j=1,\ldots,n,\>\> k=1,\ldots,p,\\
\sum_{j=1}^{n} \sum_{k=1}^{p} x_{jkt} &=& 1 \quad\quad t=1,\ldots,T,\\
\sum_{t \in \mathcal{T}_{ab}} x_{jkt} - \sum_{t \in \mathcal{T}_{ab}} x_{j,k+1,t+1} &\geq& 0 \quad\quad j=1,\ldots,n, \>\> k=1,\ldots,p-1, \\
&&\quad\quad\>\>a=1,\ldots,n, \>\> b=1,\ldots,p, \nonumber\\
x_{jkt} & \in & \{0,1\} \quad\quad j=1,\ldots,n,\>\> k=1,\ldots,p,\\
&&\quad\quad\quad\quad\>\> t=1,\ldots,T. \nonumber
\end{eqnarray}

Constraints (1) assure that each job part is assigned to exactly one time interval and constraints (2) state that exactly one job part must be assigned to each time interval. Constraints (3) state that the job parts have to be processed in the right order and that the number of time intervals between two subsequent job parts is either zero or a multiple of $p$. Note that $\mathcal{T}_{ab} = \emptyset$ when both $b=p$ and $a =n$, since the final time interval $T$ does not need to be a part of $\mathcal{T}_{ab}$: only the $p$-th part of a job can be assigned to it. Constraints (4) are Boolean constraints. \\

\begin{wraptable}{r}{0cm}
\begin{tabular}{c"c|c|c|c|c|c"}
 & 1 & 2 & 3 & 4 & 5 & 6 \\ \thickhline
1.1 & 0 & 0 & 0 & 0 & $\infty$ & $\infty$\\ \hline
1.2 & $\infty$ & 0 & 0 & 0 & 0 & $\infty$\\ \hline
1.3 & $\infty$ & $\infty$ & 3 & 4 & 5 & 6 \\ \thickhline
2.1 & $\infty$ & $\infty$ & 0 & 0 & $\infty$ & $\infty$\\ \hline
2.2 & $\infty$ & $\infty$ & $\infty$ & 0 & 0 & $\infty$\\ \hline
2.3 & $\infty$ & $\infty$ & $\infty$ & $\infty$ & 15 & 18 \\ \thickhline
\end{tabular}
\caption{Weight vector representation}
\end{wraptable}

For better understanding of the model, consider the following example: $n=2$ jobs with common processing time $p=3$, release times $r_1=1$, $r_2 = 3$ and weights $w_1=1$, $w_2=3$. The weights $w_{jkt}$ can be represented as a square matrix with $np$ rows and $np$ columns(Table 1). Each row corresponds to a job-part and each column corresponds to a time interval. The value of $w_{jkt}$ can be found in the cell defined by row $j.k$ and column $t$. 

In our model the numbering of job-parts is fixed and a job-part $j.k_2$ can not be processed before job part $j.k_1$ if $k_1 < k_2$. This ordering implies that any job $j$'s $k$-th job-part $j.k$ can not be processed before the time interval $r_j + k - 1$ and after the time interval $np - k + 1$.

In this example, job 2 becomes available for processing only at $r_2 = 3$, therefore weights $w_{2.1.1}$ and $w_{2.1.2}$ are assigned $\infty$. Then, for instance, job-part $2.2$ can not be processed before time interval 4, because the previous part $2.1$ has to be processed beforehand and the earliest it can happen is $r_2=3$. And from the other end, job-part $2.2$ can not be processed after time interval 5, because the following part $2.3$ can only be processed after $2.2$ and $2.3$ is processed not later than time interval 6.

Because only the last job-part contributes to the objective function, all other weights for job-parts $1, \ldots, p-1$ are 0. And for the last job-part $p$ in column (time interval) $t$ the cost $w_{jpt} = w_j t$ of completing job $j$ at the moment (time interval) $t$.

In programming, sum of every other number in the weight vector can used as $\infty$. The value should be greater than the optimal objective function value. Then, in terms of the scheduling problem, that value is great enough to never be considered and, in a sense, is infinite.

Note that a regular Integer Linear Programming (ILP) model is obtained when constraints (4) are replaced by \[x_{jkt} \in \mathbb{Z}_{+}\] and the BLP's Linear Programming (LP) relaxation is obtained when they are replaced by \[x_{jkt} \geq 0.\]

The model has $npT = n^2p^2$ variables and $np + T + n^2p(p-1) = O(n^2p^2)$ constraints. LP models are polynomially solvable, but since its input size depends on parameter $p$, that model is only pseudo-polynomially solvable in terms of the original problem. On the other hand, ILP and BLP models are known to be NP-complete.

\section{Algorithms}
\label{sec:4}

In this section we describe algorithms, which convert an LP solution into a Boolean one. The algorithms do not guarantee preservation of the objective function value, hence they are heuristics. In our computational study we will also involve the WSRPT heuristic (\cite{WSRPT}) which is applied to the problem $1|pmtn;r_j|\sum w_j C_j$ with arbitrary processing times. The WSRPT heuristic is one of the best known heuristics for the problem $1|pmtn;r_j|\sum w_j C_j$. Therefore, it will be useful not only in search of an optimal solution, but also for comparison with our heuristics.

\subsection{Heuristics' description}
In this subsection we describe variations of heuristics presented in \cite{article} and \cite{Bouma} as exact algorithms. 

\newpage

\textbf{Algorithm 1}
\begin{enumerate}
    \item Solve LP relaxation, obtaining the solution $x^\star$ and the optimal value $l_b$ of objective function, which will serve as a lower bound.
    \item Determine $J_I$ --- a set of jobs, for which all related variables in $x^\star$ are integral, and $J_F$ -- a set of jobs, which have at least one fractional related variable. (Both sets may be empty)
    \item Schedule every job in $J_I$ as suggested by $x^\star$.
    \item For every job $j$ in $J_F$ determine an ''estimated'' completion time as the greatest $t$, for which $x_{jpt} > 0$. 
    \item Schedule jobs in $J_F$ in the ascending order of their estimated completion times to the earliest possible time intervals. If several jobs have equal estimated completion times, choose the one that has the greatest weight. If there still are several, choose the one with the smallest index.
\end{enumerate}

\textbf{Algorithm 2}
\begin{enumerate}
    \item Solve LP relaxation, obtaining the solution $x^\star$ and the optimal value $l_b$ of objective function, which will serve as a lower bound.
    \item Determine $J_I$ --- a set of jobs, for which all related variables in $x^\star$ are integral, and $J_F$ -- a set of jobs, which have at least one fractional related variable. (Both sets may be empty)
    \item Schedule every job in $J_I$ as suggested by $x^\star$.
    \item Find the smallest $t$, for which there exists a job $j$, not yet scheduled, such that $x_{jpt} > 0$. If there are multiple such jobs, choose the one with the smallest index. If there are none, then go to step 6. 
    \item If it is possible to schedule job $j$ not later than time $t$, schedule $j$ to earliest possible time intervals. If not, remove $x_{jpt}$ from consideration. Return to step 4.
    \item If there still are unscheduled jobs, taking all $x_{jpt}$ into consideration again, find the smallest $t$, for which there exists an unscheduled job $j$ such that $x_{jpt} > 0$, and schedule job $j$ to the earliest possible time intervals. Repeat until all jobs are scheduled.
\end{enumerate}

First step of the algorithms is the most computationally expensive. Although it is hard to express the time necessary for modern industrial solvers to find a solution to the LP problem with asymptotics, it is at least $\Omega (m^2)$, where $m$ is the number of variables in the LP problem. As we have $n^2 p^2$ variables in our problem, solving LP problem will take at least $\Omega (n^4 p^4)$ time. 

Second step checks every variable, of which there are $n^2 p^2$. Third step can also easily be done with at most $O(n^2 p^2)$ operations. For the algorithm 1, step 4 can be done in $O(nT)$ time; step 5 requires sorting at most $n$ completion times, which can be done in $O(n \log n)$ time, and scheduling, which can be done in $O(T)$ time. For the algorithm 2, steps 4, 5 and 6 can be easily done in $O(npT)$ operations with the most straightforward implementation. Since $T=np$, without step 1 both algorithms can be executed in $O(n^2 p^2)$ time. It is possible to make computation more efficient at certain steps, but that is not of particular importance, because solving LP relaxation is much more time consuming, and these improvements will bear insignificant time gains. 

\subsection{An example}

\begin{wraptable}{l}{0cm}
\begin{tabular}{l"c c c c}
$j$ & 1 & 2 & 3 & 4 \\ \thickhline
$r_j$ & 1 & 4 & 3 & 2 \\ 
$w_j$ & 4 & 9 & 12 & 9 \\ 
\end{tabular}
\caption{Jobs' characteristics}
\end{wraptable}

To clarify, consider an example. The example consists of $n=4$ jobs with common processing time $p=2$ and release times and weights presented in Table 2. Now, we will go through Algorithms' steps to see how they work.

Both algorithms start by solving an LP relaxation. Corresponding LP solution is presented in the Table 3. The lower bound found by the LP relaxation is $l_b = 182$.

\begin{wraptable}{r}{0cm}
\begin{tabular}{c"c|c|c|c|c|c|c|c"}
 & 1 & 2 & 3 & 4 & 5 & 6 & 7 & 8 \\ \thickhline
1.1 & 1.0 & & & & & & &\\ \hline
1.2 & & \cellcolor{blue!50}{0.5} & & & & & & \cellcolor{red!50}{0.5} \\ \thickhline
2.1 & & & & & 0.5 & & 0.5 &\\ \hline
2.2 & & & & & & \cellcolor{blue!50}{0.5} & & \cellcolor{red!50}{0.5} \\ \thickhline
3.1 & & & 0.5 & 0.5 & & & &\\ \hline
3.2 & & & & \cellcolor{blue!50}{0.5} &\cellcolor{red!50}{0.5} & & &\\ \thickhline
4.1 & & 0.5 & & & & 0.5 & &\\ \hline
4.2 & & & \cellcolor{blue!50}{0.5} & & & & \cellcolor{red!50}{0.5} &\\ \thickhline
\end{tabular}
\caption{LP relaxation solution}
\end{wraptable}

Next, both algorithms start by distinguishing between integral and fractional jobs. Then they both schedule integral ones. In this example, there are no integral jobs and these steps are skipped in both algorithms. Following steps are different for the algorithms, so we will consider them separately.

First, consider algorithm 1. On step 4 it determines estimated completion time by finding the latest non-zero variable, associated with the last job part. These are marked in red in the table. Then it sorts the jobs by these estimated completion times and by weights, if completion times are equal. That results in the order of jobs $(3, 4, 2, 1)$. Then, it schedules the jobs in that order at the earliest possible time intervals. First, job 3 is scheduled at time intervals 3 and 4. Then job 4 is scheduled at intervals 2 and 5. Then job 2 is scheduled at intervals 6 and 7. And job 1 is left with 1 and 8. The resulting schedule is $(14334221)$. Objective function for that schedule is $4 \cdot 8 + 9 \cdot 7 + 12 \cdot 4 + 9 \cdot 5 = 188$, which is more that $l_b=182$. That example shows that Algorithm 1 is a heuristic, rather than an exact algorithm.

Moving on to Algorithm 2. On step 4, the algorithm finds the smallest $t$, for which there is a non-zero variable, associated with the last job-part of a yet unscheduled job. Initially, it's $t=2$ for $j=1$. The algorithm moves to step 5. It is possible to schedule job 1, hence, it is scheduled to the earliest possible time intervals --- 1 and 2. Then, algorithm returns to step 4 and finds next smallest $t=3$ and corresponding $j=4$. But it is not possible to schedule job 4 before time interval 3. Next earliest non-zero completion variable belongs to job 3 at $t=5$ and it is possible to schedule it before that time. Therefore, algorithm schedules it to the earliest possible time intervals 3 and 4. Note that the completion time can be earlier than the smallest $t$ found. Next smallest $t=6$ for $j=2$ is possible to schedule and job 2 is scheduled at intervals 5 and 6. There's only one unscheduled job left, which also has an associated with last job-part non-zero decision variable at $t=8$. But even if it didn't, it would have been scheduled to intervals 7 and 8 on step 6. The resulting schedule is $(11334422)$ and its objective function is $4 \cdot 2 + 9 \cdot 8 + 12 \cdot 4 + 9 \cdot 6 = 182$. Which is equal to the lower bound, hence, the solution is optimal. 

The WSRPT heuristic would construct schedule $(14433221)$, which is also optimal. 

\subsection{Intuitions behind heuristics}

Note that the algorithms are different in how they select jobs to schedule, but similar in how they schedule selected jobs. Both algorithms try to infer completion times from the solution of the LP relaxation and build a schedule with such completion times. The idea to somehow infer completion times of the jobs and then to schedule them by assigning them to earliest possible time intervals is actually sensible. Because given a set of completion times, we can easily construct a schedule with such completion times, if that is possible. 

Closely related Deadline Scheduling Problem $1|prmtn;r_j,p_j,D_j|F$ is to decide if there is a schedule such that each task can be completely executed within the interval of its release time and deadline. The algorithm to check, if a solution for $1|prmtn;r_j,p_j,D_j|F$ exists, is described in \cite{horn}. It is often referred to as Earliest Deadline First (EDF) algorithm. Basically, it schedules jobs in order of their deadlines to earliest possible time intervals. Feasible schedule exists if and only if the EDF algorithm successfully finishes. And if it exists, it is built by the EDF algorithm.

It is easy to modify algorithm for building a schedule for predetermined completion times instead of deadlines. We will assign last job-part of each job to its specified completion time. Then we will use the EDF algorithm with completion times used as deadlines and processing times less by one on that partially assigned schedule. It is easy to see that similar result holds true: a feasible solution for such completion times exists if and only if the modified EDF algorithm successfully finishes. And if it exists, it is built by the modified EDF algorithm.

That explains why the previous heuristics might be effective. Because if we infer completion times correctly, then we can easily build a schedule for them with the modified EDF algorithm --- if at all possible. And the part of converting estimated completion times to actual schedules is very similar in both Algorithm 1 and Algorithm 2. But the completion time inference part is different, allowing the algorithms to compliment each other. 

\begin{table}[]
\makeatletter
\@fpsep\textheight
\makeatother
\begin{tabular}{|r|r|r|r|r|r|r|r|r|r|}\hline
$n$ & $p$ & Instances & WSRPT & Int LP & Alg1 & Alg2 & Mean pmtn & Time, s. \\ \thickhline
\begin{filecontents*}{Tests.csv}
n,p,$r_j$,tests,wsrpt,mean pmtn,max pmtn,mean gap,max gap,LP integral,Alg1,mean pmtn,max pmtn,mean gap,max gap,Alg2,mean pmtn,max pmtn,mean gap,max gap,min time,max time,mean time,min LP time,max LP time,mean LP time,mean opt pmtn,max opt pmtn,sev nodes
10,2,$U[1; 2(4)]$,1000000,967458,0.250423,4,0.000278957,0.0558376,999670,999945,0.233438,4,1.44E-06,0.0765661,999981,0.233259,3,1.85E-07,0.0307692,0.002,0.058,0.004,0.002,0.058,0.00381578,0.2,3,0
50,2,$U[1; 2(44)]$,10000,4540,4.1973,11,0.000330222,0.00607021,9960,9995,3.6921,10,3.99E-07,0.00170659,9997,3.6907,10,2.96E-07,0.00137963,0.159,0.558,0.28,0.159,0.557,0.27549,3.6,10,0
100,2,$U[1; 2(94)]$,2000,324,9.6,21,0.000194295,0.00607021,1973,1993,8.296,18,2.96E-07,0.000177991,1996,8.289,18,9.56E-08,6.77E-05,1.379,3.835,2,1.401,3.835,1.98876,8.2,18,2
150,2,$U[1; 2(144)]$,500,34,14.8,24,0.000147713,0.00204459,496,500,12.73202,22,0,0,500,12.728,22,0,0,5.088,12.26,6.4,5.088,12.26,6.380914,12.7,21,0
200,2,$U[1; 2(194)]$,300,7,20.12333333,31,0.000111326,0.000858084,298,299,17.22,27,9.47E-07,0.000284212,299,17.22,27,1.14E-06,0.000343072,13.457,39.406,15.5,13.457,39.406,15.70076667,17.2,27,1
250,2,$U[1; 2(244)]$,20,0,25.25,34,6.25E-05,0.000220252,20,20,22.2,29,0,0,20,22.2,29,0,0,29.738,81.344,43.7,29.737,81.343,43.6866,22.2,29,0
300,2,$U[1; 2(294)]$,20,0,30.5,41,7.72E-05,0.000277425,20,20,27.05,35,0,0,20,27.05,35,0,0,168.736,3220.72,599,168.735,3220.69,599.447,27.1,35,0
350,2,$U[1; 2(344)]$,10,0,34.4,41,4.69E-05,0.000121484,10,10,30.6,36,0,0,10,30.6,36,0,0,473.191,2109.45,1383,473.189,2109.45,1382.77,30.5,36,0
10,3,$U[1; 3(4)]$,100000,95803,0.45798,4,0.000271582,0.0400433,99981,99996,0.43039,4,6.20E-07,0.0384472,100000,0.43029,4,0,0,0.007,2.653,0.01,0.007,2.653,0.0141757,0.4,4,0
50,3,$U[1; 3(44)]$,1000,435,5.946,13,0.000276115,0.0050619,1000,1000,5.47,12,0,0,1000,5.47,12,0,0,0.769,2.459,1.6,0.768,2.459,1.57307,5.4,12,0
100,3,$U[1; 3(94)]$,100,18,12.7,20,0.000139683,0.000824263,99,99,11.69,18,1.10E-06,0.000110182,100,11.68,18,0,0,8.716,29.974,17.7,8.716,29.973,17.6558,11.6,17,0
150,3,$U[1; 3(144)]$,20,0,21.25,31,0.000105347,0.000297726,20,20,19.1,27,0,0,20,19.1,27,0,0,46.846,107.435,86.5,46.846,107.434,86.4676,18.9,27,0
200,3,$U[1; 3(194)]$,20,0,28.3,35,7.31E-05,0.000392634,20,20,25.35,34,0,0,20,25.35,34,0,0,360.304,882.172,532,360.274,882.143,531.632,25.2,34,0
250,3,$U[1; 3(244)]$,10,0,34,45,8.88E-05,0.000235618,10,10,30.1,37,0,0,10,30.1,37,0,0,1267.1,3190.03,1835,1267.07,3190,1835.4,30.1,37,0
10,5,$U[1; 5(4)]$,100000,95466,0.67507,5,0.000233236,0.0396135,99980,99997,0.64673,5,3.62E-07,0.0195378,100000,0.64667,5,0,0,0.04,0.944,0.13,0.04,0.944,0.127312,0.6,5,0
30,5,$U[1; 5(24)]$,1000,676,3.804,10,0.000249109,0.0071876,998,1000,3.628,10,0,0,1000,3.627,10,0,0,1.635,2.611,2,1.635,2.611,1.9784,3.6,10,0
50,5,$U[1; 5(44)]$,100,45,7.4,13,0.000238121,0.00405956,100,100,6.97,12,0,0,100,6.97,12,0,0,10.301,13.121,11.5,10.301,13.121,11.4538,6.9,12,0
70,5,$U[1; 5(64)]$,100,25,10.68,18,0.000173469,0.00158933,100,100,9.91,17,0,0,100,9.91,17,0,0,32.736,44.947,36.8,32.736,44.947,36.8275,9.8,17,0
90,5,$U[1; 5(84)]$,20,5,16.4,21,0.000102602,0.0008446,19,19,15.7,21,1.81E-05,0.000362396,20,15.6,21,0,0,83.323,121.899,94.5,83.323,121.898,94.542,15.5,21,0
110,5,$U[1; 5(104)]$,20,1,18.05,26,0.000245083,0.0010702,20,20,16.75,23,0,0,20,16.75,23,0,0,170.107,2858.41,418,170.106,2858.41,417.743,16.4,23,0
130,5,$U[1; 5(124)]$,10,1,23,29,5.32E-05,0.000111071,10,10,22.1,28,0,0,10,22.1,28,0,0,375.658,1749.39,595,375.656,1749.39,594.618,22.1,28,0
150,5,$U[1; 5(144)]$,10,0,25.2,29,0.000134525,0.000321459,10,10,23.6,26,0,0,10,23.6,26,0,0,784.248,2255.92,1673,784.217,2255.89,1672.98,23.4,26,0
10,10,$U[1; 10(4)]$,10000,9548,0.8759,5,0.000204192,0.0273224,9999,9999,0.8454,5,3.96E-06,0.0395588,10000,0.8452,5,0,0,0.39,1.113,0.55,0.39,1.113,0.545478,0.8,5,0
20,10,$U[1; 10(14)]$,1000,817,2.655,8,0.000264251,0.0121628,999,1000,2.55,7,0,0,999,2.551,7,1.15E-06,0.00114519,3.872,5.686,4.7,3.872,5.686,4.66145,2.5,7,0
30,10,$U[1; 10(24)]$,100,62,4.66,9,0.000377095,0.0059768,100,100,4.49,9,0,0,100,4.49,9,0,0,14.3,18.232,16.3,14.299,18.232,16.3127,4.4,9,0
40,10,$U[1; 10(34)]$,100,48,6.72,13,0.00030873,0.0039051,99,100,6.42,13,0,0,100,6.42,13,0,0,39.215,51.886,44.6,39.215,51.885,44.6484,6.2,13,0
50,10,$U[1; 10(44)]$,20,11,8.45,15,0.000236863,0.001571,20,20,7.7,12,0,0,20,7.7,12,0,0,102.755,135.074,119,102.754,135.073,118.788,7.6,12,0
60,10,$U[1; 10(54)]$,20,9,12,16,0.000114557,0.000817991,20,20,11.15,16,0,0,20,11.15,16,0,0,199.737,242.861,223,199.736,242.86,222.761,11.1,16,0
70,10,$U[1; 10(64)]$,10,1,12.9,16,0.000141339,0.000782489,10,10,12.4,15,0,0,10,12.4,15,0,0,362.738,441.037,404,362.735,441.033,404.366,12.1,15,0
80,10,$U[1; 10(74)]$,3,0,11.6667,15,0.000317469,0.000437496,3,3,11,15,0,0,3,11,15,0,0,641.727,770.027,711,641.691,769.993,710.676,11,15,0
10,15,$U[1; 15(4)]$,1000,945,0.938,5,0.000224047,0.0131748,1000,1000,0.905,5,0,0,1000,0.905,5,0,0,1.318,2.128,1.6,1.318,2.128,1.58979,0.9,5,0
20,15,$U[1; 15(14)]$,100,87,2.77,7,0.000201471,0.00893596,100,100,2.68,7,0,0,100,2.68,7,0,0,12.196,18.175,14.8,12.196,18.175,14.8152,2.7,7,0
30,15,$U[1; 15(24)]$,100,67,5.27,11,0.000313391,0.00387646,100,100,5.18,11,0,0,100,5.18,11,0,0,48.519,88.244,57.6,48.519,88.244,57.597,5.1,11,0
40,15,$U[1; 15(34)]$,20,15,7.25,12,3.13E-05,0.000250078,20,20,6.95,12,0,0,20,6.95,12,0,0,127.296,240.192,163,127.295,240.191,162.817,7,12,0
50,15,$U[1; 15(44)]$,10,4,8.5,15,0.000295261,0.00116799,10,10,7.6,12,0,0,10,7.6,12,0,0,320.051,398.79,365,320.049,398.789,365.053,7.5,12,0
10,20,$U[1; 20(4)]$,100,95,0.92,4,0.00036391,0.0222006,100,100,0.9,4,0,0,100,0.9,4,0,0,3.117,4.466,3.7,3.117,4.466,3.72093,0.9,4,0
20,20,$U[1; 20(14)]$,100,76,2.96,7,0.000253231,0.0044061,100,100,2.86,7,0,0,100,2.86,7,0,0,30.645,40.431,35.2,30.645,40.43,35.15,2.8,7,0
30,20,$U[1; 20(24)]$,20,14,5.05,10,0.000133512,0.00103108,20,20,5.15,8,0,0,20,5.15,8,0,0,130.385,161.58,149,130.385,161.579,149.178,6,8,0
40,20,$U[1; 20(34)]$,3,2,7.33333,9,0.000106524,0.000319572,3,3,7,8,0,0,3,7,8,0,0,405.381,427.355,417,405.351,427.324,417.006,7,8,0
\end{filecontents*}
\csvreader[
column count = 40,
late after line=\\\hline,
]{Tests.csv}{
1 = \one, 2 = \two, 4 = \four, 5 = \five, 10 = \ten, 11 = \eleven, 16 = \sixteen, 23 = \twentythree, 27 = \twentyseven, 28 = \twentyeight, 
}
{\one & \two & \four & \five & \ten & \eleven & \sixteen & \twentyseven & \twentythree}
\end{tabular}
\caption{Experimental results}
\end{table}

\begin{figure}[]
    \includegraphics[height=8cm]{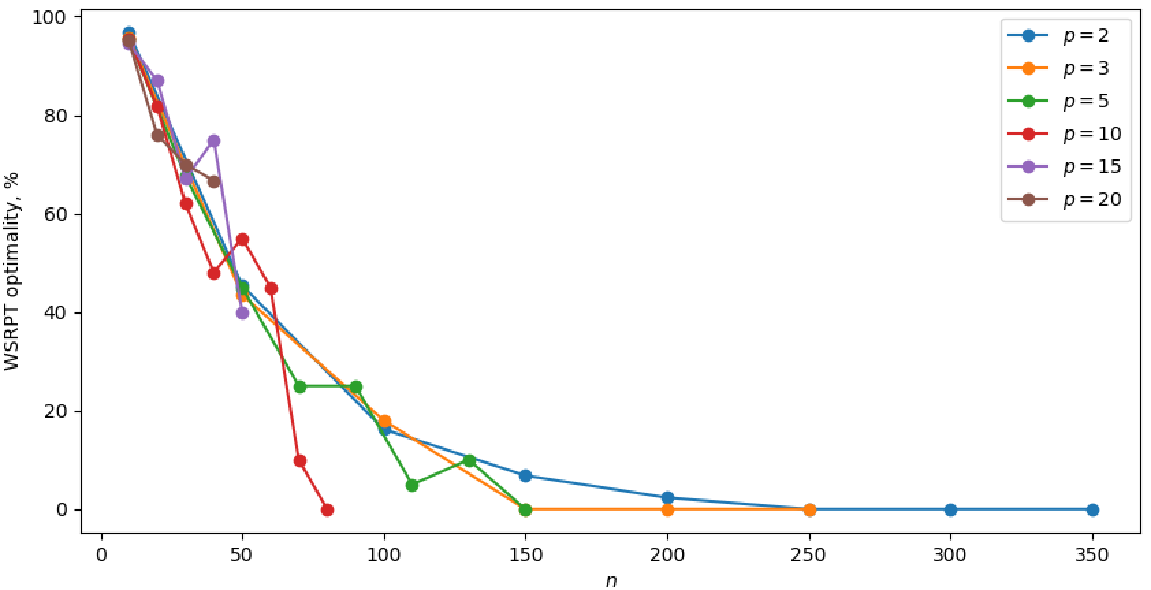}
    \caption{WSRPT optimality depending on $n$}
\end{figure}

\begin{figure}[]
    \includegraphics[height=8cm]{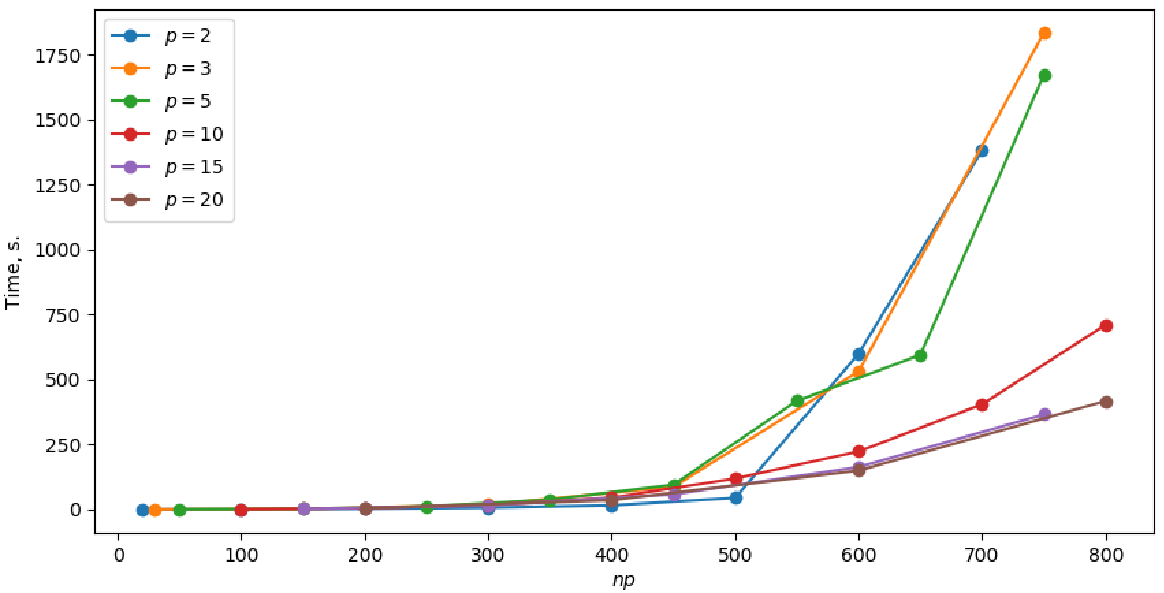}
    \caption{Time depending on problem size $np$.}
\end{figure}

\begin{figure}[]
    \includegraphics[height=8cm]{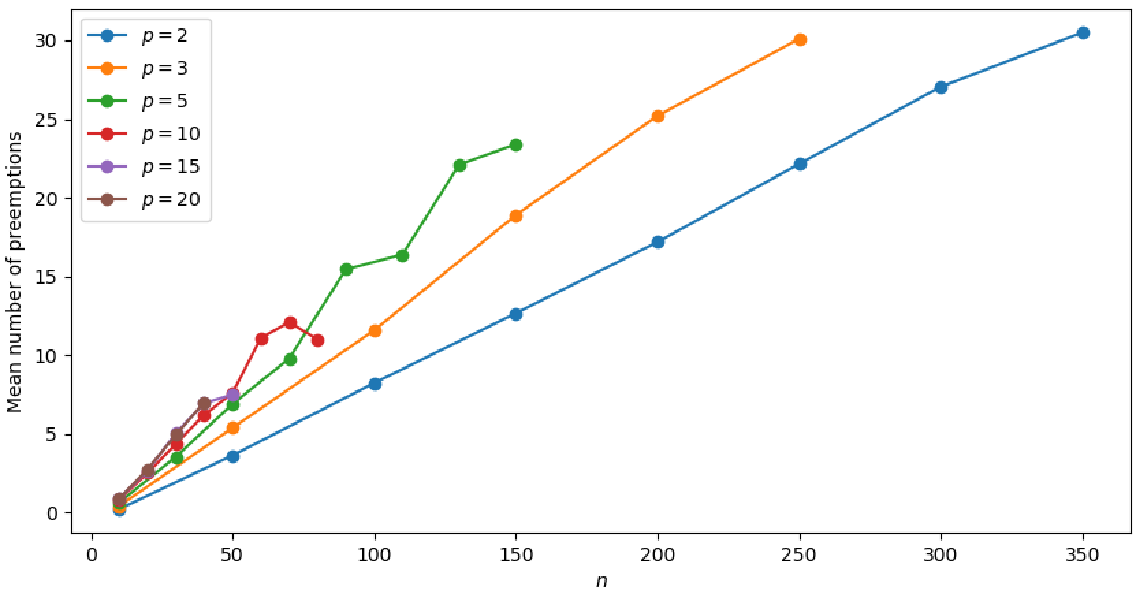}
    \caption{Mean number of preepmtions depending on $n$.}
\end{figure}

\subsection{Exact algorithm for \texorpdfstring{$1|pmtn; p_j = p; r_j| \sum w_j C_j$}{the problem}}

To find an exact solution, firstly, we solve the LP relaxation and obtain lower bound $l_b$ as the objective function optimal value. Secondly, we construct three feasible solutions using the WSRPT heuristic, Algorithm 1 and Algorithm 2. If any of them achieves $\lceil l_b \rceil$, then we've already solved the problem completely. 

If not, then we employ a Branch and Bound algorithm. It selects the largest fractional $x_{jkt}$ and considers two options: in an optimal BLP solution, $x_{jkt}$ is either 0 or 1. Options are easy to take into account with LP model by adding a corresponding constraint. Algorithm 1 and Algorithm 2 will be executed for the solution of the new LP model, possibly finding better schedule.

Branching will continue unless optimal objective value of the LP model in the node is greater or equal than the value of objective function for an already found feasible schedule. 

\section{Computational experiments}
\label{sec:5}

We performed computational experiments to evaluate the exact algorithm's and heuristics' performance. All computations were performed on a PC with i5-7300HQ processor with 3.5 GHz and 8 GB of RAM. Gurobi solver(\cite{gurobi}) was used for the LP relaxations.

More than a million instances of various sizes were randomly generated to evaluate the model's effectiveness and efficiency. For every problem size release dates were sampled from uniform distribution $U[1, p(n-6)]$ to ensure that jobs are released throughout the whole time horizon. Instances which require idle time were skipped, as they can always easily be divided into several sub-problems. Weights were sampled from uniform distribution $U[1,30]$, as 30 different priority levels are typically enough to express differences between jobs.

The results are presented in the Table 4. First column describes the number of jobs $n$, second --- common processing time $p$, third --- number of instances generated and solved, fourth --- how many instances have been solved to optimality by the WSRPT heuristic, fifth --- for how many instances the LP solution found was integral, sixth --- how many of the instances have been solved to optimality by Algorithm 1, seventh --- how many of the instances have been solved to optimality by Algorithm 2, eighth  --- mean number of preemptions in the best of the optimal solutions found, lastly, ninth --- mean running time of the whole BnB complex algorithm.

All generated instances and optimal schedules can be found at \cite{benchmark}. 

The most important observation is that solutions to the LP relaxation are very often integral. For only 449 out of more than a million instances, which is less than $0.037\%$, solutions to the LP relaxation were fractional. Integral solutions to the LP relaxation can easily be converted into feasible schedules with the same objective function value. 

That shows how powerful the property described in Corollary 2.2 is. Because the constraints (3) based on it, which are the constraints distinguishing the scheduling problem from Assignment Problem, are restrictive and powerful enough, that even without Boolean constraints (4) the relaxed model very rarely deviates from them by producing fractional solutions. That effect has not been present in testing of other models, where schedule's feasibility was expressed differently. 

The WSRPT heuristic's accuracy deteriorates with growth of $n$. Percentage of optimal schedules produced by the WSRPT heuristic depending on the number of jobs $n$ can be seen in the Fig. 5. Despite the fact that the heuristic rarely finds optimal schedules for larger problems, it does find good approximations. Mean relative optimality gap did not exceed $0.03\%$ for all problem sizes. For particular instances optimality gap can exceed $5\%$ for small problems ($n \leq 20$). For larger instances ($n \geq 30$) optimality gap did not exceed $0.7\%$. 

It has to be mentioned that the WSRPT heuristic was designed for the general problem with different processing times. The heuristic was long considered to be state-of-the-art, but could not be tested on the problems of that size, as no efficient exact algorithms were known. The WSRPT heuristic performs especially well on smaller instances. On larger instances it often fails to find an optimal schedule, but still finds a very tight approximation to it. That once more confirms the WSRPT heuristic's effectiveness, especially considering how simple it is both to understand and to implement and how computationally efficient it is. 

It is harder to measure the effectiveness of Algorithm 1 and Algorithm 2, because they are only used when the LP solution is not integral, which does not happen often. Still, it can be seen that they do turn fractional solutions into feasible schedules, which are frequently optimal. 

Except for only three instances, every other instance was solved at the root of the BnB algorithm, because at least one out of three heuristics found a solution with the objective function value equal to the LP relaxation optimal value. Relative optimality gap reached 7\% for Algorithm 1 and 3\% for Algorithm 2 in smallest instances with $n=10$ and $p=2$. But for larger instances, excluding instances with $n=10$ and instances with $n=50, p=2$, relative optimality gap was always less than $0.2\%$ for both Algorithms. 

Even though the complexity of every part of the BnB algorithm depends on $np$, rather than $n$ or $p$ separately, instances with larger $n$ took more time than problems with the same $np$, but smaller $n$. Fig.6 shows how processing time depends on $np$. Possible reason behind that could be that for the same $np$, weight matrix $w_{jkt}$ in the BLP model is denser for problems with larger $n$, because there are $n$ non-zero rows for each job's last job-part. That also makes optimal objective function value larger and forces us to choose a larger substitute for $\infty$.

Also, we want to describe the number of preemptions in the optimal solutions. Preemptions are beneficial in scheduling problems, improving the optimal schedule's objective function value by $9.21\%$ on average (\cite{pmtn_improvement}). In our problem setting, we consider preemptions to be instant and cost-free, but in real-life applications that assumption is usually incorrect. Thus, if the BnB algorithm finds several optimal schedules, it selects a schedule with the fewest preemptions as the reported schedule. Relation between the mean number of preemptions in the reported schedule and $n$ is shown in Fig. 7. 

It is only natural that the number of preemptions grows with increase in number of jobs or increase in jobs' processing times. Natural, if new jobs appear throughout the whole time horizon, at least, because, as was proven in \cite{WSRPT}, a job can only be preempted in favor of another one at the latter's release date. And we deliberately chose distribution of release dates to create more possibilities for preemptions. Otherwise, after the latest release date scheduling of not yet started jobs would become trivial, because one would just have to put them in non-decreasing order of $p_j'/w_j$, where $p_j'$ is the remaining processing time of job $j$, see e.g. \cite{properties}.

In our testing, we reached the problem size of $np=800$. Remember that computationally hardest step of the heuristics is finding a solution to the LP relaxation, which has $O(n^2 p^2)$ variables and constraints. For $np = 800$ we have more than half a million variables and constraints. It becomes time consuming to solve problems of that size on a standard personal computer. The computational time complexity of the heuristics without solving the LP relaxation and without branching, which very rarely occurs, is only $O(n^2 p^2)$. LP problems in general have received a great deal of attention: there are many industrial solvers, which can be efficiently used on powerful machines and distributed systems. Therefore solving larger problems is absolutely possible with more powerful equipment and/or more time. 

State-of-the-art results can be found in \cite{Jar_Erk_2017} for the equal processing time case or in \cite{Jar_Erk_2020} for the case with different processing times. There, processing time were sampled in $U[1, 9]$ and algorithms could not always find an optimal solution for the instances with more than 10 jobs and could never find an optimal solution for the instances with more than 20 jobs in 2 hours. For the cases, where previous algorithms found an optimal solution, our algorithm takes significantly less time. 

\section{Model flexibility}
\label{sec:6}

In this section we will discuss possible modifications to the model. We have considered one of the simplest formulations of single machine scheduling problem with common processing times. This approach can be applied to many other variations of our single machine scheduling problem. However, two requirements have to be satisfied. First, the objective function must be a non-decreasing function of completion times. Second, it must be possible to precompute the objective function values independently for each job and express it with weights in the BLP model. An example of an objective function violating the first criterion is the Total Weighted Earliness and other similar objective functions. An example of an objective function violating the second criterion is the Maximum Lateness and other similar objective functions. Nevertheless, there are still a few popular settings, where our approach is applicable. 

\begin{enumerate}
    \item Several other objective functions like Total Weighted Tardiness, Total Weighted Number of Tardy Jobs.
    \item Deadlines. If a job must be finished by a certain point in time, we can either set infinite weight for variables corresponding to post-deadline processing or add a constraint, forbidding any post-deadline processing. 
    \item Maintenance and unavailability periods. The columns corresponding to such periods can be omitted and weights in the model changed accordingly to the time omitted.
\end{enumerate}

And many others, if only they can be expressed in changes to the weights or additional constraints. The efficiency of the approach for other settings is a target of future study. 

\section{Conclusion and further research}
\label{sec:7}

In this paper we have coined the Boolean Linear Programming model for solving the scheduling problem $1|pmtn; r_j, p_j=p|\sum w_j C_j$. We have proven several important properties of optimal schedules for a wide class of scheduling problems. These properties justify and are deeply ingrained in our model. 

To find an optimal schedule, we first solve the LP relaxation of the BLP model. Very often the solution to the LP relaxation is integral, which solves the problem immediately. If the solution is fractional, we propose two heuristics, which convert it to a feasible schedule. We embed our heuristics in the BnB algorithm, which solves the problem to optimality. 

Exhaustive computational experiments with up to 350 jobs show model's performance. Size of the problems solved substantially exceeds previous state-of-the-art results (\cite{Jar_Erk_2017}, \cite{Jar_Erk_2020}). The model is scalable and can be used for larger instances with more computational power or time provided. Also, the model is quite flexible and allows for numerous modifications to suit other environments.

First direction of future research is proving --- or disproving --- integrality of the objective function value of the LP relaxation of the BLP model for the TWCT objective function. If the integrality claim is correct, then the next step is to prove --- or disprove --- that it is possible to convert a fractional solution to integral one without increase in the objective function value. If a polynomial or pseudo-polynomial algorithm converting a fractional solution into a feasible schedule exists, that would establish the scheduling problem's complexity status.

Another direction of future research is to test the model's efficiency for preemptive single machine scheduling problems with different setting, e.g., problems with different objective functions, such as total weighted tardiness or weighted number of tardy jobs, problems with deadlines, problems with unavailability periods, just to mention a few.

Lastly, the study of the preemptive nature of the problem is of great interest. Since in practice preemptions typically are not cost-free it could be beneficial to minimize their number. Beyond minimization, it would be interesting to learn more about the "price of preemptions", i.e. the difference between the objective function values of optimal preemptive and non-preemptive schedules.


\bibliographystyle{dinat}
\bibliography{bib.bib}

\end{document}